\documentclass[10pt,a4paper]{article}
\usepackage[utf8]{inputenc}
\usepackage[T1]{fontenc}
\usepackage[english]{babel}
\usepackage{amsmath}
\usepackage{amsfonts}
\usepackage{amssymb}
\usepackage{amsthm}
\usepackage{makeidx}
\usepackage{graphicx}
\usepackage{xcolor}
\usepackage{tikz}

\theoremstyle{theorem}
\newtheorem{theorem}{Theorem}
\newtheorem{lemma}{Lemma}
\newtheorem{corollary}{Corollary}
\newtheorem{proposition}{Proposition}
\newtheorem{conjecture}{Conjecture}
\newtheorem*{result}{Result}

\theoremstyle{definition}
\newtheorem{definition}{Definition}
\newtheorem*{remark}{Remark}
\newtheorem{example}{Example}

\begin{document}
	
	\title{\textbf{Influence of different kind of thin sets in the theory of convergence}}
	\author{Manoranjan Singha, Ujjal Kumar Hom}
	
	\date{}
	\maketitle
	{\let\thefootnote\relax\footnotetext{{MSC: Primary 40A35, Secondary 54A20.\\Department of Mathematics, University of North Bengal, Raja Rammohunpur, Darjeeling-734013, West Bengal, India.\\ Email address: manoranjan.math@nbu.ac.in, rs\_ujjal@nbu.ac.in}}}
	\begin{abstract}
		A class of subsets designated as very thin subsets of natural numbers has been studied and seen that theory of convergence may be rediscovered if very thin sets are given to play main role instead of thin or finite sets which removes some drawback of statistical convergence. While developing the theory of very thin sets, concepts of super thin, very very thin and super super thin sets are evolved spontaneously.
		
	\end{abstract}

	\noindent

	\section{Introduction}
	Let's begin with the well-known definition of asymptotic density \textbf{\cite{u12}} of subsets of set of natural numbers $\omega$. For any $A\subset \omega$, $|A|$ denotes the cardinality of $A$ and $A(n)=|\{m\in \omega:m\in A\cap\{1,2,...,n\} \}|$. The numbers
	\begin{center}
		$\underline{d}(A) = \displaystyle{\liminf_{n\rightarrow\infty}} \frac{A(n)}{n}$ and $\overline{d}(A) = \displaystyle{\limsup_{n\rightarrow\infty}} \frac{A(n)}{n}$
	\end{center}
are called the lower and upper asymptotic density of A, respectively. If $\underline{d}(A) = \overline{d}(A)$, then $d(A)=\overline{d}(A)$ is called asymptotic density of $A$. As in \textbf{\cite{u13}}, $A$ is called thin subset of $\omega$ if $d(A)=0$ otherwise $A$ is nonthin. 

The concept of statistical convergence \textbf{\cite{u4}} of a real sequences, is a generalization of usual convergence, is based on notion of asymptotic density where thin subsets of $\omega$ play an important role. A sequence $(x_n)_{n\in \omega}$ of real numbers is statistically convergent to a real number $a$ if for any $\epsilon>0$ the set $\{n\in \omega : \mid x_n-a \mid\geqslant\epsilon\}$ is thin.

Consider a real sequence $(x_n)_{n\in \omega}$ where
\begin{center}
$x_n=\begin{cases}-1,& \text{if }n=2^k+j,k\in \omega \text{ and } 0\leq j\leq k-1\\ 1, & \text{otherwise } \end{cases} $
\end{center}
In this sequence -1 is repeated at a stretch $k$ times from $(2^k)^{th}$ term to $(2^k+(k-1))^{th}$ term for every natural number $k$. As $k$ increases	towards infinity, number of repetition of -1 at a stretch is also increases towards infinity.\\
Let's consider another real sequence $(y_n)_{n\in \omega}$ where
\[y_n=\begin{cases}-1,& \text{if }n=2^k,k\in \omega\\ 1, & \text{otherwise } \end{cases} \]
In this sequence -1 appearing only at every ${2^k}^{th}$ place for every $k\in\omega$. As $k$ increases towards infinity gap between two consecutive appearance of -1 is also tending to infinity. In the existing literature both $(x_n)_{n\in \omega}$ and $(y_n)_{n\in \omega}$ are statistically convergent. This article distinguishes these kinds of sequences in regard of convergence.   
		\section{Definitions of different kind of thin subsets of $\omega$}\label{sec1}
			Let $A=\{n_{1}<n_{2}<n_{3}<...\}$ be an infinite subset of $\omega$ such that $b_k=(n_{k+1}-n_k)>m$ eventually for any $m\in \omega$. Then  $\displaystyle{\lim_{k \to \infty}} \frac{\frac{1}{b_1}+\frac{1}{b_2}+...+\frac{1}{b_k}}{k} = 0$. Now
			\begin{center}
				$(b_1+b_2+...+b_k)(\frac{1}{b_1}+\frac{1}{b_2}+...+\frac{1}{b_k})\geqslant k^{2} \Rightarrow\frac{\frac{1}{b_1}+\frac{1}{b_2}+...+\frac{1}{b_k}}{k}\geqslant\frac{k}{b_1+b_2+...+b_k}\geqslant\frac{k}{n_{k+1}}.$
			\end{center}
			Hence $\displaystyle{\lim_{k \to \infty}} \frac{k+1}{n_{k+1}} = 0$ and so $A$ is thin.
		\begin{example}\label{example1}
			Let $A_1=\{2^k:k\in \omega\}$ and $A_2=\{2^k+1:k\in \omega\}$. Let $\mathbb{A}=A_1\cup A_2=\{n_{1}<n_{2}<n_{3}<...\}$. Then $\mathbb{A}$ is thin but $(n_{k+1}-n_{k})=1$ frequently.
		\end{example}
		Suppose $A\subset \omega$ and $M\in \omega$. Define \newline
			$(A)_M$=$\{1\}$ $\cup$ $\{n$: $n>1$ and there exist $n$ consecutive elements of $A$ such that difference between any two consecutive among them is less than or equal to $M$$\}$.
	\begin{definition}
		A subset $A$ of $\omega$ is super thin if there exists a sub-collection $\{A_n: n\in \omega\}$ of finite subsets of $\omega$ such that $\max(A\backslash A_n)_n=1$ for all $n$.
	\end{definition}
	\begin{proposition}
		A subset $A$ of $\omega$ is super thin if and only if $A$ is finite or $\displaystyle{\lim_{n \to \infty}} (n_{k+1}-n_{k}) = \infty$ if $A=\{n_{1}<n_{2}<n_{3}<...\}$.
	\end{proposition}

     Every super thin subset of $\omega$ is thin and from Example \ref{example1} we see that thin set may not be super thin but finite union of super thin sets may not be super thin. However there exists a sub-collection $\{B_n: n\in \omega\}$ of finite subsets of $\omega$ such that $\max(\mathbb{A}\backslash B_n)_n\leqslant 2$ for all $n$ and using this fact from Example we have set the next definition.
	\begin{definition}
		A subset $A$ of $\omega$ is very thin if there exist a sub-collection $\{A_n: n\in \omega\}$ of finite subsets of $\omega$ and $M\in \omega$ such that $\max(A\backslash A_n)_n\leqslant M$ for all $n$.
	\end{definition}
    \begin{proposition}\label{Prop1}
    	A subset $A$ of $\omega$ is very thin if and only if $A$ is finite or $A$ can be written as follows:\newline $(V1)$ $A=\displaystyle{\bigcup_{k\in\omega}}\mathcal{A}_k$ where $1\leqslant |\mathcal{A}_k|\leqslant \mathcal{M} $ for some $\mathcal{M}\in\omega$ for all $k\in\omega$ ,\newline $(V2)$ $\min(\mathcal{A}_{k+1}) - \max(\mathcal{A}_k) >0$ for all $k\in \omega$,\newline $(V3)$ $\displaystyle{\lim_{k \to \infty}}(\min(\mathcal{A}_{k+1}) - \max(\mathcal{A}_k))=\infty$.
    \end{proposition}
\begin{proof}
	For any finite subset $A$ of $\omega$, $\max(A)_n\leqslant |A|$ for all $n$. So finite subsets of $\omega$ are very thin.\\
	Let $A$ be an infinite very thin subset of $\omega$. Then we can get a sub-collection $\{A_n: n\in \omega\}$ of non-empty finite subsets of $\omega$ and $M\in \omega$ such that $(\max(A_{n+1})-\max(A_n))>2^n$ and $\max(A\backslash A_n)_{t_n}\leqslant M$ for all $n$ where $\{t_{1}<t_{2}<t_{3}<...\}$ is an infinite subset of $\omega$. If $|B_k=A\cap \{\max(A_k)+1,...,\max(A_{k+1})\}|>1$ then $B_k$ can be decomposed as \begin{center}
		$B_k=\displaystyle{\bigcup_{i=1}^{j_k}}B_{ki}$
	\end{center}such that $1\leqslant |B_{ki}|\leqslant M$, $1\leqslant i\leqslant j_k$ and
	\begin{center}
		$\min(B_{k(i+1)}) - \max(B_{ki}) >t_k$, $1\leqslant i\leqslant j_k-1$.
	\end{center}
Thus $A$ can be expressed in such way that $A$ satisfies $(V1)$, $(V2)$ and $(V3)$ where $\mathcal{M}=\max\{M,\max{A_1}\}$.\\
Conversely, let $A$ be an infinite subset of $\omega$ so that $A$ satisfies $(V1)$, $(V2)$ and $(V3)$. Then there exists an infinite subset $\{n_1<n_2<n_3<...\}$ of $\omega$ such that 
\begin{center}
	$(\min(\mathcal{A}_{m+1}) - \max(\mathcal{A}_m))>k$ for all $m>n_k$.
\end{center}
Let $A_k=\{1,...,\min(\mathcal{A}_{n_k+1})\}$, $k\geqslant 1$. Then $\max(A\backslash A_k)_k\leqslant \mathcal{M}$ for all $k$.
\end{proof}
	\begin{example}
		Let $A=\displaystyle{\bigcup_{k\in \omega}A_k}$ where $A_k=\{2^k,2^k+1,...,2^k+k\}$. Then $d_n(A)\leqslant \frac{(k+1)(k+2)}{2^k}$ for $2^k\leqslant n<2^{k+1}$. So, $A$ is thin. If $A=\displaystyle{\bigcup_{k\in \omega}B_k}$ where $1\leqslant |B_k|\leqslant M $ for all $k\in\omega$ for some $M\in\omega$ and $\min(B_{k+1}) - \max(B_k) >0$ for all $k\in \omega$ then there exists a subset $\{n_1<n_2<n_3<...\}$ of $\omega$ such that $\displaystyle{\lim_{k \to \infty}}(\min(B_{n_{k}+1}) - \max(B_{n_k}))=1$. Therefore, $A$ is not very thin.\\
		
		The Prime number theorem implies that set of prime numbers is thin (see in \textbf{\cite{u9}}). Now we will see whether set of prime numbers is very thin or not. As in \textbf{\cite{u8}}, a set $\mathcal{D}=\{d_1,...,d_k\}$ consisting of non-negative integers is called admissible set if for any prime $p$, there is an integer $b_{p}$ such that $b_p\not\equiv$ d(mod p) for all $d\in \mathcal{D}$.
		The statement of Prime k-tuple conjecture is given as in \textbf{\cite{u8}}:
		\begin{conjecture}[\textbf{Prime k-tuple conjecture}]
			Let $\mathcal{D}=\{d_1,...,d_k\}$ be an admissible set. Then there are infinitely many integers $h$ such that $\{h+d_1,...,h+d_{k}\}$ is a set of primes.
		\end{conjecture}
	\begin{result}
		Set of prime numbers is not very thin.
	\end{result}
\begin{proof}
		Let $P$ be the set of all primes and let $P=\displaystyle{\bigcup_{k\in\omega}}A_k$ where $1\leqslant |A_k|\leqslant M $ for all $k\in\omega$ for some $M\in\omega$ and  $\min(A_{k+1}) - \max(A_k) >0$ for all $k\in \omega$. Then there exists an admissible set $\{d_1,...d_{M+1}\}$. \\Let $G=\max\{d_{i+1}-d_i:i=1,...,M\}$. By k-tuple conjecture there should be infinitely many ($M$+1)-tuple of primes $(p+d_1,...,p+d_{M+1})$. Therefore $(\min(A_{k+1}) - \max(A_k))\leq G$ for infinitely many $k\in \omega$. So $P$ is not very thin.
		\end{proof}
		Now the question is whether finite union of very thin sets is very thin or not. We will get this answer in the next section. Before that we have more two definitions of namely super super thin and very very thin which are generalized from super thin and very thin 
		respectively.
		
		\begin{definition}
			A subset $A$ of $\omega$ is super super thin if $A$ is finite or $\displaystyle{\sum_{k=1}^{\infty}}\frac{1}{(n_{k+1}-n_k)}<\infty$ if $A=\{n_{1}<n_{2}<n_{3}<...\}$.
		\end{definition}
		\begin{definition}
			A subset $A$ of $\omega$ is very very thin if $A$ is finite or $A$ can be written as follows:\newline (i) $A=\displaystyle{\bigcup_{k\in\omega}}\mathcal{A}_k$ where $1\leqslant|\mathcal{A}_k|\leqslant \mathcal{M} $ for some $\mathcal{M}\in\omega$ for all $k\in\omega$,\newline (ii) $\min(\mathcal{A}_{k+1}) - \max(\mathcal{A}_k) >0$ for all $k\in \omega$,\newline (iii)$\displaystyle{\sum_{k=1}^{\infty}}\frac{1}{(\min(\mathcal{A}_{k+1}) - \max(\mathcal{A}_k))}<\infty$.
		\end{definition}
		\begin{example}
			$\displaystyle{\bigcup_{k\in \omega}\{2^k,2^k+k\}}$ is super thin and very very thin but not super super thin.
		\end{example}

		\begin{example}
			Let $X=\{b_1,b_2,b_3,...\}$ and $Y=\{b_1,b_1+1,b_2,b_3,b_3+1,b_4,...\}$ where $b_k=1+...+k$. Let $X=\displaystyle{\bigcup_{k\in\omega}}X_k$ where $1\leqslant |X_k|\leqslant M $ for all $k\in\omega$ for some $M\in\omega$. Then $(\min(X_{k+1}) - \max(X_k))\leq (b_{l+1}-b_l)=l+1$ for some $l\leq kM
			$. Hence for all $k\geqslant 1$,
			\begin{center}
			 $\frac{1}{k+1}\leq \frac{M}{\min(X_{k+1}) - \max(X_k)}$.
			 \end{center}
			  Therefore $X$ is not very very thin but super thin. Since $X\subset Y$, $Y$ is not very very thin but very thin.
		\end{example}
		
		\section{Characterization of very thin sets and its relation with thin and uniformly thin sets}\label{sec2}

		\begin{lemma}\label{lemma1}
			Union of two super thin sets is very thin.
		\end{lemma}
		\begin{proof}
			Suppose $S=\{s_1<s_2<s_3<...\}$ and $T=\{t_1<t_2<t_3<...\}$ are super thin subsets of $\omega$. For each $i\in \omega$ we construct a set containing $s_i$ and the smallest number $t$ of T such that $s_i\leqslant t\leqslant \frac{s_i+s_{i+1}}{2}$ if such a number $t$ exists and  the largest number $t$ of $T$ such that $\frac{s_{i-1}+s_{i}}{2}\leqslant t\leqslant s_i$ if such a number $t$ exists. Leave all remaining elements of $T$ as singleton. Then $S\cup T$ can be decomposed into the sets $\mathcal{A}_k$ such that $S\cup T=\displaystyle{\bigcup_{k\in\omega}}\mathcal{A}_k$ where $1\leqslant |\mathcal{A}_k|\leqslant 3 $ for all $k\in\omega$ and $(\mathcal{A}_k)_{k\in \omega}$ satisfies $V(2)$ and $V(3)$ given in Proposition \ref{Prop1}.
		\end{proof}
	\begin{corollary}\label{co1}
		Union of two super super thin sets is very very thin.
	\end{corollary}
		\begin{lemma}\label{lemma2}
			Union of super thin and very thin subsets of $\omega$ is very thin.
		\end{lemma}
		\begin{proof}
			Let S be a very thin and $T=\{t_1<t_2<t_3<...\}$ be a super thin subsets of $\omega$. Let $S=\displaystyle{\bigcup_{k\in\omega}}A_k$ where $1\leqslant |A_k|\leqslant M $ for all $k\in\omega$ for some $M\in\omega$ and  $\min(A_{k+1}) - \max(A_k) >0$ for all $k\in \omega$ with $\displaystyle{\lim_{k \to \infty}}(\min(A_{k+1}) - \max(A_k))=\infty$.\newline
			 For every $i\in \omega$, we construct a set $B_i$ containing $A_i$ and the smallest number $t$ of $T$ such that $\max(A_i)\leqslant t\leqslant \frac{\max(A_i)+\min(A_{i+1})}{2}$ if such a number $t$ exists and  the largest number $t$ of $T$ such that $\frac{\min(A_{i-1})+\max(A_{i})}{2}\leqslant t\leqslant \min(A_i)$ if such a number $t$ exists and elements $t$ of $T$ such that $\min(A_i)<t<\max(A_i)$. Now each $B_i$ can be decomposed as \begin{center}
			 	$B_i=\displaystyle{\bigcup_{k=1}^{j(i)}}B_{ik}$
			 	\end{center}such that 
		 	\begin{center}
			 	$\min(B_{i(k+1)}) - \max(B_{ik}) >0$, $1\leqslant k\leqslant j(i)-1$,\\$\max(B_{ik})$ $\in T$ for $1\leqslant k< j(i)$ and $\min(B_{ik})$ $\in T$ for $1< k\leqslant j(i)$
			 \end{center}
		 and there is no $r\in \omega$ so that $t_r,t_{r+1}\in B_{ik}$ if there do not exist any $s\in A_i$ satisfying $t_r<s<t_{r+1}$. Leave all remaining elements of $T$ as singleton, one can decompose $S\cup T$ into the sets $\mathcal{A}_k$ such that $S\cup T=\displaystyle{\bigcup_{k\in\omega}}\mathcal{A}_k$ where $1\leqslant |\mathcal{A}_k|\leqslant 2M+1 $ for all $k\in\omega$ and $(\mathcal{A}_k)_{k\in \omega}$ satisfies $V(2)$ and $V(3)$ given in Proposition \ref{Prop1}. 
		\end{proof}
	\begin{corollary}\label{co2}
		Union of super super thin and very very thin subsets of $\omega$ is very very thin.
	\end{corollary}
		\begin{theorem}\label{thverythin}
			A subset A of $\omega$ is very thin if and only if A can be expressed as a finite union of super thin subsets of $\omega$.
		\end{theorem}
		\begin{proof}
			Suppose A is a very thin subset of $\omega$ such that\newline (i) $A=\displaystyle{\bigcup_{k\in\omega}}A_k$ where $1\leqslant |A_k|\leqslant M $ for all $k\in\omega$ for some $M\in\omega$,\newline (ii) $\min(A_{k+1}) - \max(A_k) >0$ for all $k\in \omega$,\newline (iii)$\displaystyle{\lim_{k \to \infty}}(\min(A_{k+1}) - \max(A_k))=\infty$.\newline
			Let $A_k=\{a_{k1}\leqslant a_{k2}\leqslant ...\leqslant a_{kM}\}, k\in \omega$. Define $B_i=\{a_{ki}:k\in \omega\}, 1\leqslant i\leqslant M$. Then $A=\displaystyle{\bigcup_{i=1}^{M}} B_i$ and
			\begin{align*}
				(a_{(k+1)i} - a_{ki})&\geqslant (a_{k+1)1} - a_{kM}) = (\min(A_{k+1}) - \max(A_k)) 
			\end{align*}
			Therefore $B_i$ is super thin, $1\leqslant i\leqslant M$.\newline Converse part follows directly from Lemma 1 and Lemma 2.
		\end{proof}
		
		\begin{corollary}
			Finite union of very thin subsets of $\omega$ is very thin.
		\end{corollary}
		\begin{corollary}\label{cor}
			Very thin subsets of $\omega$ are thin.
		\end{corollary}
		
		\end{example}
		\begin{corollary}
			Any very thin set can be expressed as a finite intersection of thin but non very thin sets.
		\end{corollary}
		\begin{proof}
			Suppose $S=\{t_1<t_2<t_3<...\}$ is a super thin subset of $\omega$. Let $n_1\in\omega$ such that $t_{n_1+1}-t_{n_1}>2$. Then we get a strictly increasing sequence of natural numbers $(n_k)_{k\geqslant 1}$ such that $t_{n_k}>2t_{n_{k-1}}$ and $t_{n_k+1}-t_{n_k}>2k$ for $k\geq2$.\newline
			Define $A=\displaystyle{\bigcup_{k\geqslant 1}\{t_{n_k},...,t_{n_k}+k\}}$ and $B=\displaystyle{\bigcup_{k\geqslant 1}\{t_{n_k+1}-k,...,t_{n_k+1}\}}$.\\Then A and B both are thin but not very thin. Let $A'=A\cup S$ and $B'=B\cup S$. Since $S$ is super thin, $A'$ and $B'$ are thin but not very thin with $A'\cap B'=S$. Hence any super thin set can be expressed as intersection of two thin but non very thin sets. From Theorem~\ref{thverythin} it follows that any very thin set can be expressed as a finite intersection of thin but non very thin sets.
		\end{proof}
		\begin{theorem}
			A subset A of $\omega$ is very very thin if and only if A can be expressed as a finite union of super super thin 
			subsets of $\omega$.
		\end{theorem}
	\begin{proof}
		If $A$ is very very thin then each $B_i$ defined in Theorem \ref{thverythin} become super super thin.\\
		Corollary \ref{co1} and Corollary \ref{co2} together implies that finite union of super super thin sets is very very thin.
	\end{proof}
\begin{corollary}
	Finite union of very very thin subsets of $\omega$ is very very thin.
\end{corollary}
	
		From Corollary \ref{cor} it follows that very thin sets are thin. Since zero uniform density\textbf{(\cite{u1,u11})} subsets of $\omega$ are thin, it is natural to arise a question that whether very thin sets have uniform density zero or zero uniform density sets are very thin.
		Let $B\subset \omega$. For $h\geqslant 0$ and $k\geqslant 1$, let $A(h+1,h+k)=|\{n\in B:h+1\leqslant n\leqslant h+k\}|$. The existence of the following limits are proved in \textbf{\cite{u11}}:
		\begin{center}
			$ \underline{u}(B) = \displaystyle{\lim_{k\to \infty}\frac{1}{k}\liminf_{h\to \infty}A(h+1,h+k)}$ , $ \overline{u}(B) = \displaystyle{\lim_{k\to \infty}\frac{1}{k}\limsup_{h\to \infty}A(h+1,h+k)}$ 
		\end{center}
		$\underline{u}(B)$ and $\overline{u}(B)$ are called lower and upper uniform density of B respectively. If $\underline{u}(B)=\overline{u}(B)$ then $u(B)=\overline{u}(B)=\underline{u}(B)$ is called uniform density of $B$. From now on call B is uniformly thin if $u(B)=0$. Since $\underline{u}(B)\leqslant \underline{d}(B)\leqslant \overline{d}(B)\leqslant \overline{u}(B)$, $B$ is thin if $B$ is uniformly thin but the converse is not true (see in \textbf{\cite{u11}}).\newline
		Now we will see the relation between very thin and uniformly thin subsets of $\omega$.
		\begin{theorem}
			Any very thin subset of $\omega$ is uniformly thin.
		\end{theorem}
		\begin{proof}
			Suppose $A$ be a very thin subset of $\omega$ such that $A=\displaystyle{\bigcup_{k\in\omega}}A_k$ where $1\leqslant |A_k|\leqslant M $ for all $k\in\omega$ for some $M\in\omega$, $\min(A_{k+1}) - \max(A_k) >0$ for all $k\in \omega$ and $\displaystyle{\lim_{k \to \infty}}(\min(A_{k+1}) - \max(A_k))=\infty$. Let $n_k=\min(A_{k+1}) - \max(A_k)$, $k\in \omega$.\newline
			Let $k_1$ = least element of $\omega$ such that $n_{k_1}\leqslant n_k$ for all $k\in \omega$ and let\newline 
			$k_{r+1}$ = least element of $\omega-\{k_1,...,k_r\}$ such that $n_{k_{r+1}}\leqslant n_k$ for all $k\in \omega-\{k_1,...,k_r\}$, $r>1$.\newline Then $\displaystyle{\lim_{r\to \infty}n_{k_r}}=\infty$ and so $\displaystyle{\lim_{r\to \infty}\frac{r}{n_{k_1}+...+n_{k_r}}}=0$. \newline 
			By induction it can be shown that for any $l\geqslant 1$ and for any  $l$ distinct elements $m_1,...,m_l$ of $\omega$,
			\begin{center}
				$n_{k_1}+...+n_{k_l}\leqslant n_{m_1}+...+n_{m_l}$.
			\end{center}
			Let $s_l=n_{k_1}+...+n_{k_l}$. Then $\{n\in A:t+1\leqslant n\leqslant t+s_l+1\}$ can intersect at most $l$ consecutive $A_k$, $k\in \omega$. Therefore
			\begin{center}
			 $\displaystyle{\max_{t\geqslant 0 }|\{n\in A:t+1\leqslant n\leqslant t+s_l+1\}|}\leqslant Ml$
			 \end{center}
		  and so $u(A)=0$.
		\end{proof}
		\begin{example}\label{ex6}
			Let $a_1=1$ and let $a_p=a_{p-1}+2(1^3+...+(p-1)^3)+1,p\geqslant 2$. \newline 
			Define $A_p=\{a_p,a_p+1^3,a_p+1^3+2^3,...,a_p+1^3+2^3+...+p^3\}$, $p\geqslant 1$. Let 
			$A=\displaystyle{\bigcup_{p=1}^{\infty}}A_p$. Then $A$ is not very thin.\newline
		Let $b_n=1^3+2^3+...+n^3$ and $s_m=|\{r\in A:m+1\leqslant r\leqslant m+b_n+1\}|$, $n\geqslant 1$, $m\geqslant 0$.\newline
			 Then
			\begin{center}
			 $s_m\leqslant (n+1)$ if $a_n\leqslant m$.
			 \end{center}
			Now $s_m\leqslant |A_1|+...+|A_n|\leqslant \frac{(n+1)(n+2)}{2}$ if $0\leqslant m<a_n$.\newline 
			Hence $A$ is uniformly thin.
		\end{example}
	\begin{figure}[!h]
		\centering
		\includegraphics[width=0.6\linewidth]{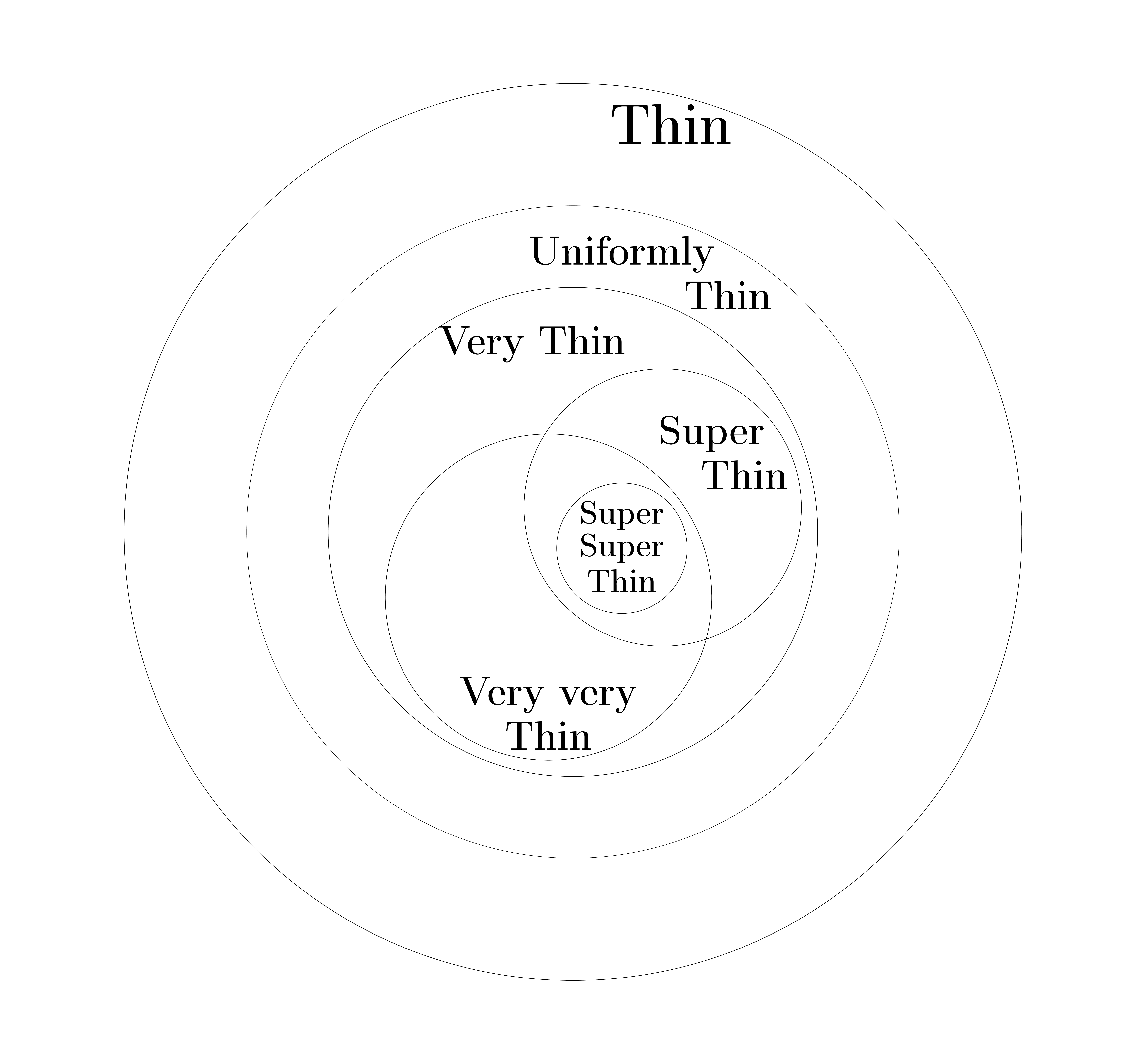}
		\caption{}
		\label{fig:figure333m}
	\end{figure}

		Example {\pageref{ex6}} shows that uniformly thin set mat not be very thin. A diagram is given in Figure~\ref{fig:figure333m} on the basis of all examples given in Section \ref{sec1} and Section \ref{sec2}.
		\section{Fin-BW and BW properties}
		 First, we recall some properties and propositions related to ideal and ideal convergence given in \textbf{\cite{u1,u5}}:\newline
		An ideal $\mathcal{I}$ on $\omega$ has the Fin-BW property if for any bounded sequence $(x_n)_{n\in \omega}$ of real numbers there is $A\notin \mathcal{I}$ such that $(x_n)_{n\in A}$ is convergent and the BW property if for any bounded sequence $(x_n)_{n\in \omega}$ of real numbers there is $A\notin \mathcal{I}$ such that $(x_n)_{n\in A}$ is $\mathcal{I}$-convergent. If an ideal $\mathcal{I}$ has Fin-BW property then it also has BW property.
		
		Recall that $2^\omega,2^{<\omega}$ and $2^n$ denote the set of all infinite sequences of zeros and ones, the set of all finite sequences of zeros and ones and sequences of zeros and ones of length $n$ respectively. If $s\in 2^n$, then $s\string^i$ denotes the sequence of length $n+1$ which extends $s$ by $i$ for $i\in \omega$. If $x\in 2^{\omega}$ then $x\upharpoonright n=(x(0),x(1),...,x(n-1))$ for $n\in \omega$.
		\begin{proposition}[\textbf{\cite{u1}}]\label{prop1}
			An ideal has the BW property (the Fin-BW property) if and only if for every family of sets $\{A_s:s\in 2^{<\omega}
			\}$ satisfying the following conditions \newline
			$(S1) A_\emptyset=\omega,\newline
			(S2) A_s=A_{s\string^0}	\cup A_{s\string^1},\newline
			(S3) A_{s\string^0}	\cap A_{s\string^1}=\emptyset$,\newline
			there exist $x\in {2^\omega}$ and $B\subset \omega$, $B\notin I$ such that $B\backslash A_{x\upharpoonright n}\in I$ ($B\backslash A_{x\upharpoonright n}$ is finite respectively) for all $n$.
		\end{proposition}
	$\mathcal{I}_d$ = ideal of all thin subsets of $\omega$ and $\mathcal{I}_u$ = ideal of all uniformly thin subsets of $\omega$ do not satisfy BW and so FinBW property ( see Example 4 in \textbf{\cite{u3}} and Corollary 1 in \textbf{\cite{u1}}). Now the following example shows that $\mathcal{I}_v$ = ideal of very thin subsets of $\omega$ does not satisfy Fin-BW property.

		\begin{example}
			Define a family of sets $\{A_s:s\in 2^{<\omega}\}$ as follows:
			\begin{center}
				$A_\emptyset=\omega$,
			\end{center}
		\begin{figure}[!ht]
			\centering
			\includegraphics[width=1\linewidth]{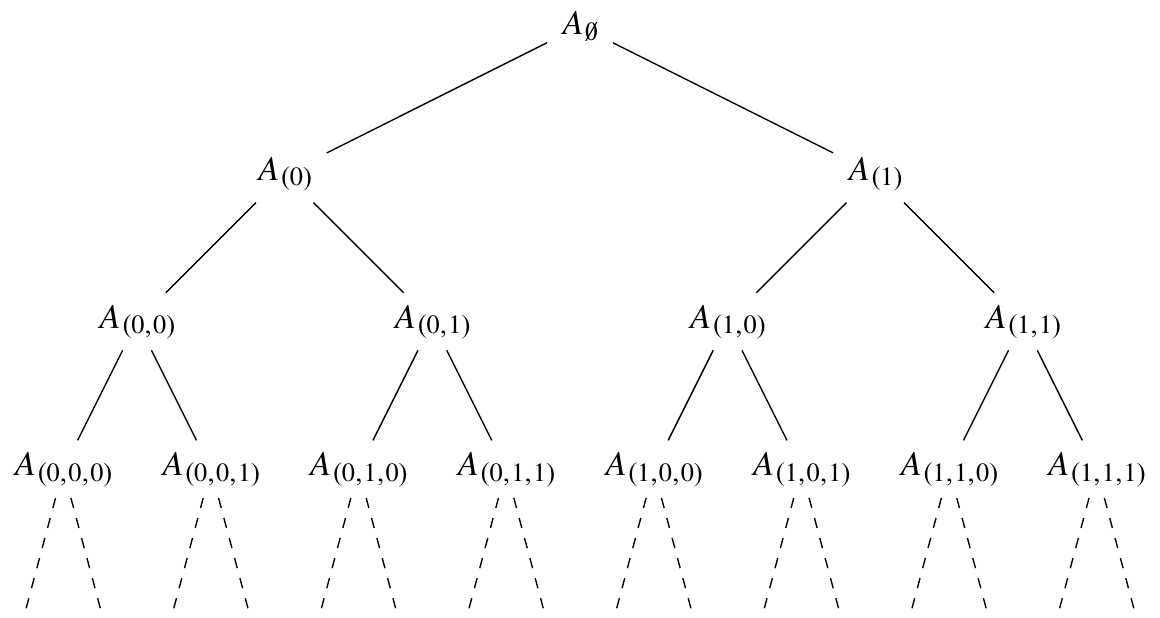}
			\caption{}
			\label{fig:figure111}
		\end{figure}
		
			\begin{center}
				$A_{(0)}=2\omega$, $A_{(1)}=2\omega-1$,
			\end{center}
			\begin{center}
				$A_{(0,0)}=2^2\omega$, $A_{(0,1)}=2^2\omega-2$, $A_{(1,0)}=2^2\omega-1$ ,$A_{(1,1)}=2^2\omega-3$,
			\end{center}
			\begin{center}
				$A_{(0,0,0)}=2^3\omega$, $A_{(0,0,1)}=2^3\omega-4$, $A_{(0,1,0)}=2^3\omega-2$, $A_{(0,1,1)}=2^3\omega-6$, \newline $A_{(1,0,0)}=2^3\omega-1$, $A_{(1,0,1)}=2^3\omega-5$, $A_{(1,1,0)}=2^3\omega-3$, $A_{(1,1,1)}=2^3\omega-7$
			\end{center}
			and so on.\newline
			So if $s\in 2^n$ then $A_s=2^n\omega-i$, $0\leqslant i\leqslant 2^n-1$ and
			\begin{center}
				$A_{s\string^0}=2^n(2\omega)-i$, $A_{s\string^1}=2^n(2\omega-1)-i$.
			\end{center}
			Let $x\in 2^{\omega}$. Then $\displaystyle{\bigcap_{n\in \omega}}A_{x\upharpoonright n}$=$\emptyset$ and $\{A_{x\upharpoonright n}\backslash A_{x\upharpoonright {n+1}}:n\in \omega\}$ is a collection of mutually disjoint sets so that numbers in  each $A_{x\upharpoonright n}\backslash A_{x\upharpoonright {n+1}}$ are in arithmetic progressions.\newline
			Let $N=\{n_0<n_1<n_2<n_3<...\}$ be an infinite subset of $\omega$ and let $a_{n_k}$ be the $n_k^{th}$ element in the arithmetic progression formed by elements of $A_{x\upharpoonright k}\backslash A_{x\upharpoonright {k+1}}$, $k\geqslant 0$.\newline
			Let $Ar_N=\displaystyle{\bigcup_{k\geqslant 0}}$\{first $n_k$ numbers in the arithmetic progression of elements of $A_{x\upharpoonright k}\backslash A_{x\upharpoonright {k+1}}$\}. Then $Ar_N$ is super thin.\newline
			Suppose $B\subset \omega$ such that $B\backslash A_{x\upharpoonright n}$ is finite for all $n$. Since $\displaystyle{\bigcap_{n\in \omega}}A_{x\upharpoonright n}$=$\emptyset$, $B\subset Ar_N$ for some infinite subset $N$ of $\omega$ and so $B$ is super thin.
		\end{example}
		\begin{theorem}
			\label{thbw}
			$\mathcal{I}_v$ = ideal comprising of very thin subsets of $\omega$ satisfies the BW property.
		\end{theorem}
		\begin{proof}
			Let $A$ be a subset of $\omega$ and $M\in \omega$. Define\newline
					$(A)_M$=$\{1\}$$\cup$$\{n\in \omega$: $n>1$ and there exist $n$ consecutive elements of $A$ such that difference between any two consecutive among them is less than or equal to $M$$\}$.
		Then $A$ is very thin implies $(A)_M$ is finite for all $M\in \omega$.\newline
			Let $\{A_s:s\in 2^{<\omega}\}$ be a family of sets satisfying the three conditions in Proposition~\ref{prop1}. This theorem can be proved in the following cases.\newline
			\newline
			$\textbf{Case 1:}$ Suppose there exist  $x\in 2^{\omega}$  and $M\in \omega$ such that $(A_{x\upharpoonright n})_M$ is infinite for all $n\in \omega$.\newline
			Let $B_0$= $\{1\}$ and define $B_{n+1}$ to be the $(n+1)$ consecutive elements of $A_{x\upharpoonright{ n+1}}$ such that difference between each two consecutive among them is $\leqslant M$ so that $\max(B_n)<\min(B_{n+1})$, $n\in \omega$.\newline
			Let $B=\displaystyle{\bigcup_{n\in \omega}}B_n$. Then $B$ is not very thin and $B\backslash A_{x\upharpoonright n}\subset\displaystyle{\bigcup_{i=0}^{n-1}}B_i$ is finite for $n>0$.\newline
			\newline
			$\textbf{Case 2:}$ Suppose for any $x\in 2^{\omega}$  and $M\in \omega$ if $A_{x\upharpoonright {n+1}}$ is not very thin for all $n\in \omega$ then there exists a $n_M\in \omega$ such that $(A_{x\upharpoonright {n_M}})_M$ is finite.\newline
			 So there exist $k_0\in \omega$ and $s_0\in 2^{k_0}$ such that $A_{s_0}$ is not very thin and $(A_{s_0})_1$ is finite and let $M_0$=$\max(A_{s_0})_1$.\newline
			As $\max(A_{s_0})_1$=$M_0$, due to the hypothesis
			there exist a $k_1\in \omega$ with $k_0<k_1$ and $s_1\in 2^{k_1}$ such that\newline 
			(i) $A_{s_0}$ and $A_{s_1}$ are disjoint and 
			$(A_{s_1})_{M_0+1}$ is finite,\newline
			(ii) there is an infinite subset $\mathcal{A}_1$ of $A_{s_0}$ such that if $p\in \mathcal{A}_1$ then $p+i\in A_{s_1}$ for some $i$, $1\leqslant i\leqslant M_0$ and if $p<q<p+i$ then $q\in A_{s_0}$,\newline(iii) $\max(A_{s_0}\cup A_{s_1})_1\leqslant (M_0+1)M_1+M_0$ = $N_1$ where $M_1$= $\max(A_{s_1})_{M_0+1}$.\newline
			Let $B_1=\{a_1^0,a_1^1\}$ where $a_1^i\in A_{s_i}$, $i\in \{0,1\}$ and $(a_1^1-a_1^0)\leqslant M_0$.\newline\newline
			Continuing in this way we will get an infinite subset $\{k_0<k_1<k_2<...\}$ of $\omega$, a sequence $(s_i)_{i\in \omega}$ such that $s_i\in 2^{k_i}$, mutually disjoint sets $\{A_{s_i}:i\in \omega\}$, a collection of infinite sets $\{A_{s_0}=\mathcal{A}_{0}\supset \mathcal{A}_{1}\supset \mathcal{A}_{2}\supset \mathcal{A}_{3}\supset ...\}$, a collection of finite sets $\{B_i:i\geqslant 1\}$, a sequence $(M_i)_{i\in \omega}$ and an infinite subset $\{M_0=N_0<N_1<N_2<...\}$ of $\omega$ such that for $n\in \omega$ \newline
			(i) $M_{n+1}$= $\max(A_{s_{n+1}})_{N_n+1}$,\newline \newline
			(ii) $\mathcal{A}_{n+1}$ infinite subset of $\mathcal{A}_{n}$ such that if $p\in \mathcal{A}_{n+1}$ then $p+i_0+i_1+...+i_k\in A_{s_{k+1}}$ for some $i_k$ with $k+1\leqslant i_0+i_1+...+i_k\leqslant N_k$, $0\leqslant k\leqslant n$ so that if $p<q<p+i_0$ then $q\in A_{s_0}$ and if $p+i_0+i_1+...+i_{k-1}< q< p+i_0+i_1+...+i_k$ then $q\in A_{s_0}\cup A_{s_1}\cup ...\cup A_{s_k}$ where $1\leqslant k\leqslant n$,\newline \newline
			(iii)  $\max(A_{s_0}\cup A_{s_1}\cup ...\cup A_{s_{n+1}})_1\leqslant (N_n+1)M_{n+1}+N_n$ $=N_{n+1}$ and\newline \newline
			(iv) $B_{n+1}=\{a_{n+1}^0,a_{n+1}^1,...,a_{n+1}^{n+1}\}$ where $a_{n+1}^i\in A_{s_i}$ and $(a_{n+1}^{i+1}-a_{n+1}^i)\leqslant N_i$ for $0\leqslant i\leqslant n$ and for $n\geqslant 1$, $a_{n+1}^0-a_{n}^n>2^{n}$.
			\begin{figure}[h!]
				\centering
				\includegraphics[width=0.86\linewidth]{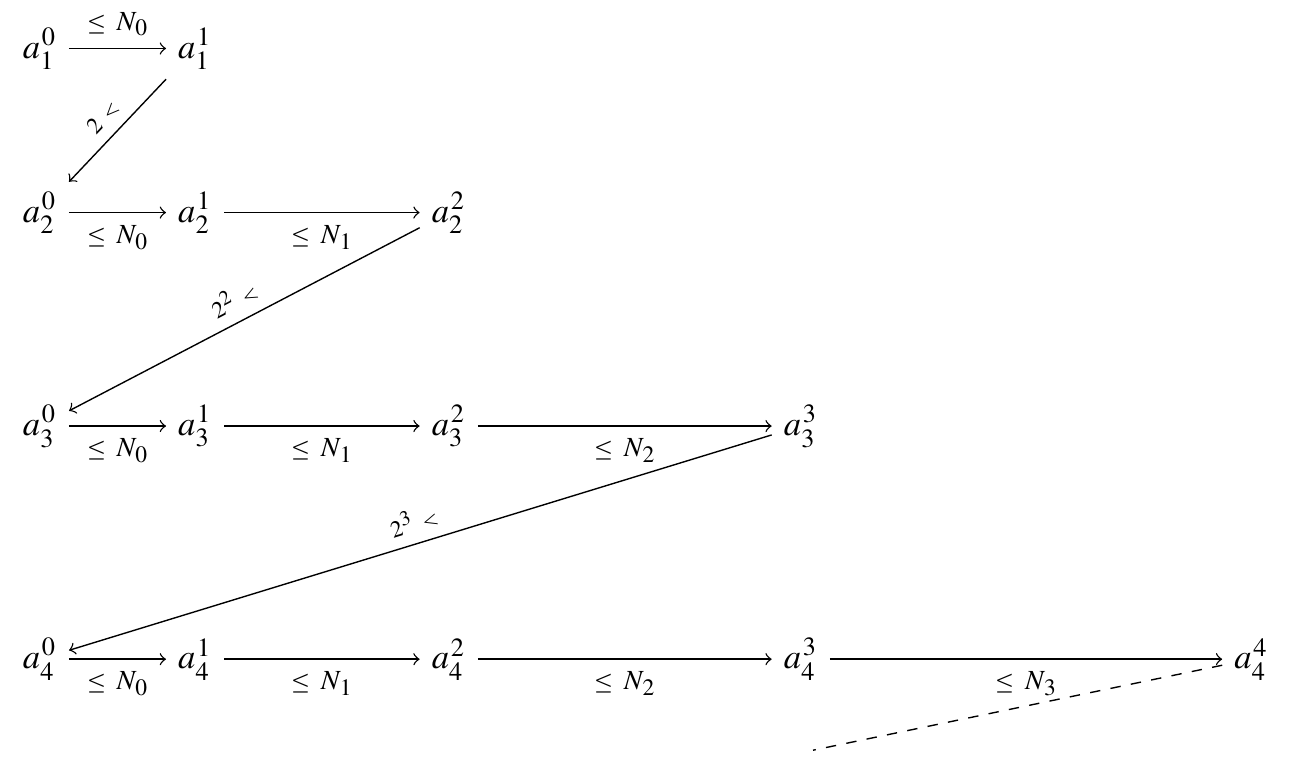}
				\caption{}
				\label{fig:figure222}
			\end{figure}\newline
			Let $B=\displaystyle{\bigcup_{n=1}^{\infty}}B_n$. Then $B$ is not very thin and $B\cap A_{s_i}$ is super thin for all $i\in \omega$. Moreover if $M$ is an infinite subset of $\{s_i:i\in \omega\}$ then $B_M=\displaystyle{\bigcup_{s\in M}}B\cap A_s$ is not very thin (see Figure~\ref{fig:figure222}).\newline
			Define $x\in 2^{\omega}$ such that there are infinitely many $i\in \omega$ so that $A_{s_i}\subset A_{x\upharpoonright n}$, $n\in \omega$.
			If there exists a $m\in \omega$ such that there is no $i\in \omega$ so that $A_{s_i}\subset A_{x\upharpoonright n}\backslash A_{x\upharpoonright {n+1}}$ for all $n\geqslant m$ then we simply take $M=\{s_i:A_{s_i}\subset A_{x\upharpoonright {m+1}}\}$. If there does not exist such $m$, construct $M$ by taking least $s_i$ such that  $A_{s_i}\subset A_{x\upharpoonright n}\backslash A_{x\upharpoonright {n+1}}$ (if such $s_i$ exists) for $n\in \omega$. In both cases $M$ is infinite and so $B_M$ is not very thin with $B_M\backslash A_{x\upharpoonright n}$ is very thin for all $n$.	
		\end{proof}

		\begin{remark}
			Let 
			\begin{center}
				$r_n=0+1+...+n$, $n\in \omega$,\\ $p_n=N_0+N_1+...+N_{r_n}$, $n\in \omega$ and\\$q_0=N_0$ and $q_n=q_{n-1}+p_n$, $n\geqslant 1$
			\end{center}
			Define a set $D=\displaystyle{\bigcup_{i=0}^{\infty}}D_{i}$ as follows:
			\begin{center}
				$D_0$=$\displaystyle{\bigcup_{i=1}^{q_0}}B_{i}$
			\end{center}
		and for $n\geqslant 1$,
		  \begin{center}
		  	$D_n$=$\displaystyle{\bigcup_{i=q_{n-1}+1}^{q_n}}B_{i}$$\bigg \backslash$ $\displaystyle{\bigcup_{i=0}^{n-1}}A_{s_i}$
		  \end{center}
			Then $D$ is not very very thin and also $D\cap A_{s_i}$ is finite for all $i\in \omega$. Similarly one can construct a non very very thin set $D_M$ such that $D\cap A_{s}$ is finite for all $s\in M$ and $D_M\subset B_M$ where $M$ is an infinite subset of $\{s_i:i\in \omega\}$.\newline
   Replacing $B_M$ by $D_M$ and defining same $x\in 2^{\omega}$ as in the last case of Theorem ~\ref{thbw} , it can be shown that $\mathcal{I}_{vv}$ = ideal of very very thin subsets of $\omega$ has the Fin-BW property. \newline\newline
		\end{remark}

		\vfill\eject
		
	\end{document}